\documentclass[11pt,thmsa,arabtex]{article}
\usepackage{amssymb}
\usepackage{amsfonts}
\usepackage{amsmath}
\usepackage{graphicx}
\usepackage{hyperref}
\usepackage{cite}
\usepackage{ntheorem}

\newtheorem*{theorem-non}{Theorem}
\newtheorem{theorem}{Theorem}

\newtheorem{proposition}{Proposition}[section]
\newtheorem{definition}{Definition}[section]
\newtheorem{notation}{Notation}[section]
\newtheorem{corollary}{Corollary}[section]
\newtheorem{example}{Example}[section]

\newenvironment{proof}[1][Proof]{\noindent\textbf{#1.} }{\ \rule{0.5em}{0.5em}}

\numberwithin{theorem}{section}
\numberwithin{equation}{section}

\begin{document}

\title{A study on special curves of AW($k$)-type\\
in the pseudo-Galilean space}
\author{\emph{H.S.Abdel-Aziz} and \emph{M.Khalifa Saad}\thanks{
~E-mail address:~mohamed\_khalifa77@science.sohag.edu.eg} \\
%EndAName
{\small Math. Dept., Faculty of Science, Sohag Univ., 82524 Sohag, Egypt}}
\date{}
\maketitle

\textbf{Abstract.} This paper is devoted to the study of AW$(k)$-type $%
\left( 1\leq k\leq 3\right) $ curves according to the equiform differential
geometry of the pseudo-Galilean space $G_{3}^{1}$. We show that equiform
Bertrand curves are circular helices or isotropic circles of $G_{3}^{1}$.
Also, there are equiform Bertrand curves of AW$(3)$ and weak AW$(3)$-types.
Moreover, we give the relations between the equiform curvatures of these
curves. Finally, examples of some special curves are given and plotted.

\bigskip \textbf{\emph{M.S.C. 2010}:} 53A04, 53A35, 53C40.


\section{Introduction}

As it is well known, geometry of space is associated with mathematical
group. The idea of invariance of geometry under transformation group may
imply that, on some spacetimes of maximum symmetry there should be a
principle of relativity which requires the invariance of physical laws
without gravity under transformations among inertial systems \cite{1}.
Besides, theory of curves and the curves of constant curvature in the
equiform differential geometry of the isotropic spaces $I_{3}^{1}$ , $%
I_{3}^{2}$ and the Galilean space $G_{3}$ are described in \cite{2} and \cite%
{3}, respectively. The pseudo-Galilean space is one of the real Cayley-Klein
spaces. It has projective signature $(0,0,+,-)$ according to \cite{2}. The
absolute of the pseudo-Galilean space is an ordered triple $\{w,f,I\}$ where
$w$ is the ideal plane, $f$ a line in $w$ and $I$ is the fixed hyperbolic
involution of the points of $f$. In \cite{4}, from the differential geometric point of
view, K. Arslan and A. West defined the notion of AW(k)-type submanifolds. Since
then, many works have been done related to AW(k)-type submanifolds (see, for
example, \cite{5,6,7,8,9,10}). In \cite{9}, \"{O}zg\"{u}r and Gezgin studied a Bertrand curve of AW$(k)$-type and furthermore, they showed that there is no such Bertrand curve of AW$(1)$ and AW$(3)$%
-types if and only if it is a right circular helix. In addition, they
studied weak AW$(2)$-type and AW$(3)$-type conical geodesic curves in
Euclidean 3-space $E^{3}$. Besides, In $3$-dimensional Galilean space and
Lorentz space, the curves of AW$(k)$-type were investigated in \cite{6,8}.
In \cite{7}, the authors gave curvature conditions and characterizations
related to AW$(k)$-type curves in $E^{n}$ and in \cite{10}, the authors
investigated curves of AW$(k)$-type in the $3$-dimensional null cone.

In this paper, to the best of author's knowledge, Bertrand curves of AW$(k)$%
-type have not been presented in the equiform geometry of the
pseudo-Galilean space $G_{3}^{1}$ in depth. Thus, the study is proposed to
serve such a need. Our paper is organized as follows. In Section $2$, the
basic notions and properties of a pseudo-Galilean geometry are reviewed. In
Section $3$, properties of the equiform geometry of the pseudo-Galilean
space $G_{3}^{1}$ are given. Section $4$ contains a study of AW$(k)$-type
equiform Frenet curves. Equiform Bertrand curves of AW$(k)$-type in $G_{3}^{1}$ included in
section $5$.

\section{Pseudo-Galilean geometric meanings}

In this section, let us first recall basic notions from pseudo-Galilean
geometry \cite{11,12}. In the inhomogeneous affine coordinates for points and vectors
(point pairs) the similarity group $H_{8}$ of $G_{3}^{1}$ has the following
form
\begin{align}
\bar{x}& =a+b.x,  \notag \\
\bar{y}& =c+d.x+r.\cosh \theta .y+r.\sinh \theta .z,  \notag \\
\bar{z}& =e+f.x+r.\sinh \theta .y+r.\cosh \theta .z,
\end{align}%
where $a,b,c,d,e,f,r$ and $\theta $ are real numbers.
Particularly, for $b=r=1,$ the group $(2.1)$ becomes the group $B_{6}\subset
H_{8}$ of isometries (proper motions) of the pseudo-Galilean space $G_{3}^{1}
$. The motion group leaves invariant the absolute figure and defines the
other invariants of this geometry. It has the following form
\begin{align}
\bar{x}& =a+x,  \notag \\
\bar{y}& =c+d.x+\cosh \theta .y+\sinh \theta .z,  \notag \\
\bar{z}& =e+f.x+\sinh \theta .y+\cosh \theta .z.
\end{align}%
According to the motion group in the pseudo-Galilean space, there are
non-isotropic vectors $A(A_{1},A_{2},A_{3})$ (for which holds $A_{1}\neq 0$)
and four types of isotropic vectors: spacelike ($A_{1}=0,$ $%
A_{2}^{2}-A_{3}^{2}>0$), timelike ($A_{1}=0,$ $A_{2}^{2}-A_{3}^{2}<0$) and
two types of lightlike vectors ($A_{1}=0,A_{2}=\pm A_{3}$). The scalar
product of two vectors $u=(u_{1},u_{2},u_{3})$ and $v=(v_{1},v_{2},v_{3})$
in $G_{3}^{1}$ is defined by
\begin{equation*}
\left\langle u,v\right\rangle =\left\{
\begin{array}{c}
u_{1}v_{1},\text{ \ \ \ \ \ \ \ \ \ \ \ \ \ \ \ \ \ \ \ \ \ \ \ \ if }%
u_{1}\neq 0\text{ or }v_{1}\neq 0, \\
u_{2}v_{2}-u_{3}v_{3}\text{ \ \ \ \ \ \ \ \ \ \ \ \ \ \ \ \ \ \ \ \ if \ }%
u_{1}=0\text{ and }v_{1}=0.%
\end{array}%
\right\}
\end{equation*}%
We introduce a pseudo-Galilean cross product in the following way
\begin{equation*}
u\times _{G_{3}^{1}}v=\left\vert
\begin{array}{ccc}
0 & -j & k \\
u_{1} & u_{2} & u_{3} \\
v_{1} & v_{2} & v_{3}%
\end{array}%
\right\vert ,
\end{equation*}%
where $j=(0,1,0)$ and $k=(0,0,1)$ are unit spacelike and timelike vectors,
respectively. Let us recall basic facts about curves in $G_{3}^{1}$, that
were introduced in \cite{15}.

A curve $\gamma (s)=(x(s),y(s),z(s))$ is called an admissible curve if it has no
inflection points $(\dot{\gamma}\times \ddot{\gamma}\neq 0)$ and no
isotropic tangents $(\dot{x}\neq 0)$ or normals whose projections on the
absolute plane would be lightlike vectors $(\dot{y}\neq \pm \dot{z})$. An
admissible curve in $G_{3}^{1}$ is an analogue of a regular curve in
Euclidean space \cite{12}.

For an admissible curve $\gamma :I\subseteq \mathbb{R}\rightarrow G_{3}^{1},$
the curvature $\kappa (s)$ and torsion $\tau (s)$ are defined by%
\begin{equation}
\kappa (s)=\frac{\sqrt{\left\vert \ddot{y}(s)^{2}-\ddot{z}(s)^{2}\right\vert
}}{(\dot{x}(s))^{2}},\text{ }\tau (s)=\frac{\ddot{y}(s)\dddot{z}(s)-\dddot{y}%
(s)\ddot{z}(s)}{\left\vert \dot{x}(s)\right\vert ^{5}\cdot \kappa ^{2}(s)},%
\text{\ }
\end{equation}%
expressed in components. Hence, for an admissible curve $\gamma :I\subseteq
\mathbb{R}\rightarrow G_{3}^{1}$ parameterized by the arc length $s$ with
differential form $ds=dx$, given by
\begin{equation}
\gamma (x)=(x,y(x),z(x)),
\end{equation}%
the formulas $(2.3)$ have the following form
\begin{equation}
\kappa (x)=\sqrt{\left\vert y^{^{\prime \prime }}(x)^{2}-z^{^{\prime \prime
}}(x)^{2}\right\vert },\text{ }\tau (x)=\frac{y^{^{\prime \prime
}}(x)z^{^{\prime \prime \prime }}(x)-y^{^{\prime \prime \prime
}}(x)z^{^{\prime \prime }}(x)}{\kappa ^{2}(x)}.
\end{equation}%
The associated trihedron is given by
\begin{align}
\mathbf{e}_{1}& =\gamma ^{\prime }(x)=(1,y^{^{\prime }}(x),z^{^{\prime
}}(x)),  \notag \\
\mathbf{e}_{2}& =\frac{1}{\kappa (x)}\gamma ^{^{\prime \prime }}(x)=\frac{1}{%
\kappa (x)}(0,y^{^{\prime \prime }}(x),z^{^{\prime \prime }}(x)),  \notag \\
\mathbf{e}_{3}& =\frac{1}{\kappa (x)}(0,\epsilon z^{^{\prime \prime
}}(x),\epsilon y^{^{\prime \prime }}(x)),
\end{align}%
where $\epsilon =+1$ or $\epsilon =-1$, chosen by criterion det$%
(e_{1},e_{2},e_{3})=1$, that means
\begin{equation*}
\left\vert y^{^{\prime \prime }}(x)^{2}-z^{^{\prime \prime
}}(x)^{2}\right\vert =\epsilon (y^{^{\prime \prime }}(x)^{2}-z^{^{\prime
\prime }}(x)^{2})\text{.}
\end{equation*}%
The curve $\gamma $ given by $(2.4)$ is timelike (resp. spacelike) if $%
\mathbf{e}_{2}(s)$ is a spacelike (resp. timelike) vector. The principal
normal vector or simply normal is spacelike if $\epsilon =+1$ and timelike
if $\epsilon =-1$. For derivatives of the tangent $\mathbf{e}_{1}$, normal $%
\mathbf{e}_{2}$ and binormal $\mathbf{e}_{3}$ vector fields, the following
Frenet formulas in $G_{3}^{1}$ hold:
\begin{align}
\mathbf{e}_{1}^{\prime }(x)& =\kappa (x)\mathbf{e}_{2}(x),  \notag \\
\mathbf{e}_{2}^{\prime }(x)& =\tau (x)\mathbf{e}_{3}(x),  \notag \\
\mathbf{e}_{3}^{\prime }(x)& =\tau (x)\mathbf{e}_{2}(x).
\end{align}

\section{Frenet formulas according to the equiform geometry of $G_{3}^{1}$}

This section contains some important facts about equiform geometry. The
equiform differential geometry of curves in the pseudo-Galilean space $%
G_{3}^{1}$ has been described in \cite{11}. In the equiform geometry a few
specific terms will be introduced. So, let $\gamma (s):I\rightarrow
G_{3}^{1} $ be an admissible curve in the pseudo-Galilean space $G_{3}^{1}$,
the equiform parameter of $\gamma $ is defined by
\begin{equation*}
\sigma :=\int \frac{1}{\rho }ds=\int \kappa ds,
\end{equation*}%
where $\rho =\frac{1}{\kappa }$ is the radius of curvature of the curve $%
\gamma $. Then, we have
\begin{equation}
\frac{ds}{d\sigma }=\rho .
\end{equation}%
Let $h$ be a homothety with the center in the origin and the coefficient $
\mu $. If we put $\bar{\gamma}=h(\gamma )$, then it follows
\begin{equation*}
\bar{s}=\mu s\text{ \ and \ }\bar{\rho}=\mu \rho ,\text{ }
\end{equation*}%
where $\bar{s}$ is the arc-length parameter of $\bar{\gamma}$ and $\bar{%
\rho}$ the radius of curvature of this curve. Therefore, $\sigma $ is an
equiform invariant parameter of $\gamma $ \cite{11}.

\begin{notation}
The functions $\kappa $ and $\tau $ are not invariants of the homothety
group, then from $(2.3)$ it follows that $\bar{\kappa}=\frac{1}{\mu }\kappa $
and $\bar{\tau}=\frac{1}{\mu }\tau $.
\end{notation}

From now on, we define the Frenet formulas of the curve $\gamma $ with
respect to its equiform invariant parameter $\sigma $ in $G_{3}^{1}.$ The
vector
\begin{equation*}
\mathbf{T}=\frac{d\gamma }{d\sigma },
\end{equation*}%
is called a tangent vector of the curve $\gamma .$ From $(2.6)$ and $(3.1)$
we get
\begin{equation}
\mathbf{T}=\frac{d\gamma }{ds}\frac{ds}{d\sigma }=\rho \cdot \frac{d\gamma }{%
ds}=\rho \cdot \mathbf{e}_{1}.
\end{equation}%
Further, we define the principal normal vector and the binormal vector by
\begin{equation}
\mathbf{N}=\rho \cdot \mathbf{e}_{2},\text{ \ }\mathbf{B}=\rho \cdot \mathbf{%
e}_{3}.
\end{equation}%
It is easy to show that $\left\{ \mathbf{T},\mathbf{N},\mathbf{B}\right\} $
is an equiform invariant frame of $\gamma .$ On the other hand, the
derivatives of these vectors with respect to $\sigma $ are given by
\begin{equation}
\left[
\begin{array}{c}
\mathbf{T} \\
\mathbf{N} \\
\mathbf{B}%
\end{array}%
\right] ^{\prime }=\left[
\begin{array}{ccc}
\dot{\rho} & 1 & 0 \\
0 & \dot{\rho} & \rho \tau \\
0 & \rho \tau & \dot{\rho}%
\end{array}%
\right] \left[
\begin{array}{c}
\mathbf{T} \\
\mathbf{N} \\
\mathbf{B}%
\end{array}%
\right] .
\end{equation}
The functions $\mathcal{K}:I\rightarrow \mathbb{R}$ defined by $\mathcal{K}=%
\dot{\rho}$ is called the equiform curvature of the curve $\gamma $ and $%
\mathcal{T}:I\rightarrow \mathbb{R}$ defined by $\mathcal{T}=\rho \tau =%
\frac{\tau }{\kappa }$ is called the equiform torsion of this curve. In the light of this, the formulas $(3.4)$ analogous to the Frenet
formulas in the equiform geometry of the pseudo-Galilean space $G_{3}^{1}$
can be written as
\begin{equation}
\left[
\begin{array}{c}
\mathbf{T} \\
\mathbf{N} \\
\mathbf{B}%
\end{array}%
\right] ^{\prime }=\left[
\begin{array}{ccc}
\mathcal{K} & 1 & 0 \\
0 & \mathcal{K} & \mathcal{T} \\
0 & \mathcal{T} & \mathcal{K}%
\end{array}%
\right] \left[
\begin{array}{c}
\mathbf{T} \\
\mathbf{N} \\
\mathbf{B}%
\end{array}%
\right] .
\end{equation}%
The equiform parameter $\sigma =\int \kappa (s)ds$ for closed curves is
called the total curvature, and it plays an important role in global
differential geometry of Euclidean space. Also, the function $\frac{\tau }{%
\kappa }$ has been already known as a conical curvature and it also has
interesting geometric interpretation.

\begin{notation}
Let $\gamma :I\rightarrow G_{3}^{1}$ be a Frenet curve in the equiform
geometry of the $G_{3}^{1}$, the following statements are true $($ see for
details \cite{11,13} $)$:
\end{notation}

\begin{enumerate}
\item If $\gamma (s)$ is an isotropic logarithmic spiral in $G_{3}^{1}$.
Then, $\mathcal{K=}$const$.\neq 0$ and $\mathcal{T}=0,$

\item If $\gamma (s)$ is a circular helix in $G_{3}^{1}$. Then, $\mathcal{K=}%
0$ and $\mathcal{T}=$const$.\neq 0,$

\item If $\gamma (s)$ is an isotropic circle in $G_{3}^{1}$. Then, $\mathcal{%
K=}0$ and $\mathcal{T}=0.$
\end{enumerate}

\section{AW($k$)-type curves in the equiform geometry of $G_{3}^{1}$}

Let $\gamma :I\rightarrow G_{3}^{1}$ be a curve in the equiform geometry of
the pseudo-Galilean space $G_{3}^{1}$. The curve $\gamma $ is called a
Frenet curve of osculating order $l$ if its derivatives $\gamma ^{\prime
}(s),\gamma ^{\prime \prime }(s),\gamma ^{\prime \prime \prime
}(s),...,\gamma ^{(l)}(s)$ are linearly dependent and $\gamma ^{\prime
}(s),\gamma ^{\prime \prime }(s),\gamma ^{\prime \prime \prime
}(s),...,\gamma ^{(l+1)}(s)$ are no longer linearly independent for all $%
s\in I$ . To each Frenet curve of order $3$ one can associate an orthonormal
$3$-frame $\left\{ \mathbf{T},\mathbf{N},\mathbf{B}\right\} $ along $\gamma $, 
such that $\gamma ^{\prime }(s)=\frac{1}{\rho }\mathbf{T}$, called the
equiform Frenet frame (Eqs. $(3.5)$).

Now, we consider equiform Frenet curevs of osculating order $3$ in $G_{3}^{1}
$ and start with some important results.

Let $\gamma :I\rightarrow G_{3}^{1}$ be a Frenet curve in the equiform
geometry of the pseudo-Galilean space. By the use of Frenet formulas $(3.5)$%
, we obtain the higher order derivatives of $\gamma $ as follows
\begin{align*}
\gamma ^{\prime }(s)& =\frac{d\gamma }{d\sigma }\frac{d\sigma }{ds}=\frac{1}{%
\rho }\mathbf{T}, \\
\gamma ^{\prime \prime }(s)& =\frac{1}{\rho ^{2}}\mathbf{N}, \\
\gamma ^{\prime \prime \prime }(s)& =\frac{1}{\rho ^{3}}\left( -\mathcal{K}%
\mathbf{N}\mathcal{+T}\mathbf{B}\right) , \\
\gamma ^{\prime \prime \prime \prime }(s)& =\frac{1}{\rho ^{4}}[(2\mathcal{K}%
^{2}\mathcal{+\mathcal{T}}^{2}-\mathcal{K}^{\prime })\mathbf{N}+(\mathcal{T}%
^{\prime }-3\mathcal{KT})\mathbf{B}].
\end{align*}

\begin{notation}
Let us write
\begin{align}
Q_{1}& =\frac{1}{\rho ^{2}}\mathbf{N}, \\
Q_{2}& =\frac{1}{\rho ^{3}}\left( -\mathcal{K}\mathbf{N}\mathcal{+T}\mathbf{B%
}\right) , \\
Q_{3}& =\frac{1}{\rho ^{4}}[(2\mathcal{K}^{2}\mathcal{+\mathcal{T}}^{2}-%
\mathcal{K}^{\prime })\mathbf{N}+(\mathcal{T}^{\prime }-3\mathcal{KT})%
\mathbf{B}].
\end{align}
\end{notation}

\begin{notation}
$\gamma ^{\prime }(s),\gamma ^{\prime \prime }(s),\gamma ^{\prime \prime
\prime }(s)$ and $\gamma ^{\prime \prime \prime \prime }(s)$ are linearly
dependent if and only if $Q_{1},Q_{2}$ and $Q_{3}$ are linearly dependent.
\end{notation}

\begin{definition}
Frenet curves (of osculating order $3$) in the equiform geometry of the
pseudo-Galilean space $G_{3}^{1}$ are called \cite{5}:

\begin{enumerate}
\item of type equiform AW$(1)$ if they satisfy \ $Q_{3}=0,$

\item of type equiform AW$(2)$ if they satisfy $\left\Vert Q_{2}\right\Vert
^{2}$ $Q_{3}=\langle Q_{3}\left( s\right) ,Q_{2} \rangle
Q_{2} ,$

\item of type equiform AW$(3)$ if they satisfy $\left\Vert Q_{1}\right\Vert
^{2}$ $Q_{3}=\langle Q_{3} ,Q_{1}\left( s\right) \rangle
Q_{1} ,$

\item of type weak equiform AW$(2)$ if they satisfy%
\begin{equation}
Q_{3}=\left\langle Q_{3},Q_{2}^{\ast }\right\rangle Q_{2}^{\ast },
\end{equation}

\item of type weak equiform AW$\left( 3\right) $ if they satisfy%
\begin{equation}
Q_{3}=\left\langle Q_{3},Q_{1}^{\ast }\right\rangle Q_{1}^{\ast },
\end{equation}%
where%
\begin{eqnarray}
Q_{1}^{\ast } &=&\frac{Q_{1}}{\left\Vert Q_{1}\right\Vert },  \notag \\
Q_{2}^{\ast } &=&\frac{Q_{2}-\left\langle Q_{2},Q_{1}^{\ast }\right\rangle
Q_{1}^{\ast }}{\left\Vert Q_{2}-\left\langle Q_{2},Q_{1}^{\ast
}\right\rangle Q_{1}^{\ast }\right\Vert }.
\end{eqnarray}
\end{enumerate}
\end{definition}

\begin{proposition}
Let $\gamma :I\rightarrow G_{3}^{1}$ be a Frenet curve $($of osculating
order $3)$ in the equiform geometry of the pseudo-Galilean space $G_{3}^{1}$,

(i) $\gamma $ is of type weak equiform AW$(2)$ if and only if
\begin{equation}
2\mathcal{K}^{2}+\mathcal{T}^{2}-\mathcal{K}^{\prime }=0,
\end{equation}

(ii) $\gamma $ is of type weak equiform AW$(2)$ if and only if
\begin{equation}
\mathcal{T}^{\prime }-3\mathcal{KT(}s\mathcal{)}=0.
\end{equation}
\end{proposition}

\begin{proof}
According to Definition 4.1 and Notation 4.1, the proof is obvious.
\end{proof}

\begin{theorem}
Let $\gamma :I\rightarrow G_{3}^{1}$ be a Frenet curve $($of osculating
order $3)$ in the equiform geometry of the pseudo-Galilean space $G_{3}^{1}$%
. Then $\gamma $ is of type equiform AW$(2)$ if and only if%
\begin{equation*}
-\mathcal{K}^{\prime }+2\mathcal{K}^{2}+\mathcal{T}^{2}=0,
\end{equation*}%
\begin{equation}
\text{\ }3\mathcal{KT-T}^{\prime }=0.
\end{equation}
\end{theorem}

\begin{proof}
Since $\gamma $ is of type equiform AW$(2)$, then from $(4.3)$,
we obtain%
\begin{equation*}
\frac{1}{\rho ^{4}}[(2\mathcal{K}^{2}\mathcal{+\mathcal{T}}^{2}(s)-\mathcal{K%
}^{\prime })\mathbf{N}+(\mathcal{T}^{\prime }-3\mathcal{KT})\mathbf{B}]=0.
\end{equation*}%
As we know, the vectors $\mathbf{N}$ and $\mathbf{B}$ are linearly
independent, so we can write%
\begin{equation*}
2\mathcal{K}^{2}\mathcal{+\mathcal{T}}^{2}-\mathcal{K}^{\prime }=0\text{ and
\ }\mathcal{T}^{\prime }-3\mathcal{KT}=0.
\end{equation*}%
The converse statement is straightforward and therefore the proof is
completed.
\end{proof}

\begin{theorem}
Let $\gamma :I\rightarrow G_{3}^{1}$ be a Frenet curve $($of osculating
order $3)$ in the equiform geometry of the pseudo-Galilean space $G_{3}^{1}$%
. Then $\gamma $ is of type equiform AW$(2)$ if and only if%
\begin{equation}
\mathcal{K}^{2}\mathcal{T}-\mathcal{KT}^{\prime }+\mathcal{TK}^{\prime }-%
\mathcal{T}^{3}=0.
\end{equation}
\end{theorem}

\begin{proof}
Assuming that $\gamma $ is a Frenet curve in the equiform geometry of $%
G_{3}^{1}$ , then from $(4.2)$ and $(4.3)$, one can write%
\begin{align*}
Q_{2}& =a_{11}\mathbf{N}+a_{12}\mathbf{B}, \\
Q_{3}& =a_{21}\mathbf{N}+a_{22}\mathbf{B},
\end{align*}%
where $a_{11}$,$a_{12}$, $a_{21}$ and $a_{22}$ are differentiable functions.
Since $Q_{2}$ and $Q_{3}$ are linearly dependent, coefficients
determinant equals zero and hence
\begin{equation}
\left\vert
\begin{array}{cc}
a_{11} & a_{12} \\
a_{21} & a_{22}%
\end{array}%
\right\vert =0,
\end{equation}%
where
\begin{eqnarray}
a_{11} &=&\frac{-1}{\rho ^{3}}\mathcal{K},\text{ }a_{12}=\frac{1}{\rho ^{3}}%
\mathcal{T},  \notag \\
a_{21} &=&\frac{1}{\rho ^{4}}[-\mathcal{K}^{\prime }+2\mathcal{K}^{2}+%
\mathcal{T}^{2}],\text{ }  \notag \\
a_{22} &=&\frac{1}{\rho ^{4}}[-3\mathcal{KT+T}^{\prime }].
\end{eqnarray}%
From $(4.11)$ and $(4.12)$, we obtain $(4.10)$. It can be easily shown that
the converse assertion is also true.
\end{proof}

\begin{corollary}
Let $\gamma :I\rightarrow G_{3}^{1}$ be a Frenet curve (of osculating
order3) in the equiform geometry of the pseudo-Galilean space $G_{3}^{1}$,

(i) If $\gamma $ is an isotropic logarithmic spiral in $G_{3}^{1}$, then $%
\gamma $ is of equiform AW$(2)$-type curve.

(ii) If $\gamma $ is an equiform space or timelike general (circular) helix
in $G_{3}^{1}$, then it is not of  equiform AW$(k)$, weak AW$(2)$ and weak
AW$(3)$-types.
\end{corollary}

\begin{theorem}
Let $\gamma :I\rightarrow G_{3}^{1}$ be a Frenet curve (of osculating order $%
3$) in the equiform geometry of $G_{3}^{1}$. Then $\gamma $ is of equiform AW%
$\left( 3\right) $-type if and only if
\begin{equation}
\mathcal{T}^{\prime }-3\mathcal{KT}=0.
\end{equation}
\end{theorem}

\begin{proof}
Using Definition $4.1$ and Eqs. $(4.1)$ and $(4.3)$, we obtain $(4.13)$. The
converse direction is obvious, hence our Theorem is proved.
\end{proof}

\section{Bertrand curves of AW$(k)$-type}

\begin{definition}
A curve $\gamma :I\rightarrow G_{3}^{1}$ with equiform curvature $\mathcal{K}%
=0$ is called an equiform Bertrand curve if there exist a curve $\bar{\gamma}%
:I\rightarrow G_{3}^{1}$ with equiform curvature $\mathcal{\bar{K}}=0$ such
that the principal normal lines of $\gamma $ and $\bar{\gamma}$ are parallel
at the corresponding points. In this case $\bar{\gamma}$ is called an
equifrm Bertrand mate of $\gamma $ and vise versa.
\end{definition}

By Definition $5.1$, we can say that for given an equiform Bertrand pair $%
\left( \gamma ,\bar{\gamma}\right) $, there exist a functional relation $%
\bar{s}=\bar{s}(s)$ such that $\lambda (\bar{s}(s)=\lambda (s)$, then the
equiform Bertrand mate of $\gamma $ is given by

\begin{equation}
\bar{\gamma}(s)=\gamma (s)+\lambda \mathbf{N}.
\end{equation}

\begin{theorem}
If $\left( \gamma ,\bar{\gamma}\right) $ is an equiform Bertrand pair in the
equiform geometry of the pseudo-Galilean space $G_{3}^{1}$, then

\begin{description}
\item[(i)] The function $\lambda $ is constant.

\item[(ii)] $\gamma $ with non-zero constant equiform torsion is a circular
helix in $G_{3}^{1}$.

\item[(iii)] $\gamma $ with zero equiform torsion is an isotropic circle of $%
G_{3}^{1}$
\end{description}
\end{theorem}

\begin{proof}
Along $\gamma $ and $\bar{\gamma}$, let $\{\mathbf{T},\mathbf{N},\mathbf{B}\}
$ and $\{\mathbf{\bar{T}},\mathbf{\bar{N}},\mathbf{\bar{B}}\}$ be the Frenet
frames according to the equiform geometry of the pseudo-Galilean space $%
G_{3}^{1}$, respectively. Differentiate $(5.1)$ with respect to $s$, we
obtain
\begin{equation}
\mathbf{\bar{T}}=\mathbf{T}+\lambda \mathbf{N}^{\prime }+\lambda ^{\prime }%
\mathbf{N}.
\end{equation}%
By using $(3.5)$, we have
\begin{equation*}
\mathbf{\bar{T}}=\mathbf{T}+\left( \lambda \mathcal{K+\lambda }^{\prime
}\right) \mathbf{N}+\lambda \mathcal{T}\mathbf{B}\text{.}
\end{equation*}%
Since $\mathbf{\bar{N}\ }$is parallel to $\mathbf{N}$, we get
\begin{equation*}
\lambda \mathcal{K+\lambda }^{\prime }\mathcal{=}0,
\end{equation*}%
it follows that
\begin{equation*}
\lambda =const.
\end{equation*}

If $\gamma $ has a non-zero constant equiform torsion, then $\gamma $ is
characterized by
\begin{equation*}
\kappa =const.\neq 0,~\tau =const.\neq 0,
\end{equation*}%
and therefore $\tau /\kappa =const.$holds.

On the other hand, whenever $\mathcal{T}=0$, the natural equations of $%
\gamma $ is given by
\begin{equation*}
\kappa =const.\neq 0,~\tau =0,
\end{equation*}%
and so, the curve $\gamma $ is an isotropic circle in $G_{3}^{1}$ \cite{14}.
Thus the proof is completed.
\end{proof}

\begin{theorem}
If $\left( \gamma ,\bar{\gamma}\right) $ is a Bertrand pair in the equiform
geometry of the pseudo-Galilean space $G_{3}^{1}$, then the angle between
tangent vectors at corresponding points is constant.
\end{theorem}

\begin{proof}
To prove that the angle is constant, we need to show that $\left\langle
\mathbf{\bar{T}},\mathbf{T}\right\rangle ^{\prime }=0$. For this purpose
using $(3.5)$ to obtain
\begin{align}
\left\langle \mathbf{\bar{T}},\mathbf{T}\right\rangle ^{\prime }&
=\left\langle \mathbf{\bar{T}}^{\prime },\mathbf{T}\right\rangle
+\left\langle \mathbf{\bar{T}},\mathbf{T}^{\prime }\right\rangle   \notag \\
& =\left\langle \mathcal{\bar{K}}\mathbf{\bar{T}}+\mathbf{\bar{N}},\mathbf{T}%
\right\rangle +\left\langle \mathbf{\bar{T}},\mathcal{K}\mathbf{T}+\mathbf{N}%
\right\rangle   \notag \\
& =\mathcal{\tilde{K}}\left\langle \mathbf{\bar{T}},\mathbf{T}\right\rangle
+\left\langle \mathbf{\bar{N}},\mathbf{T}\right\rangle +\mathcal{K}%
\left\langle \mathbf{\bar{T}},\mathbf{T}\right\rangle   \notag \\
& \ \ \ +\left\langle \mathbf{\bar{T}},\mathbf{N}\right\rangle .\text{ \ \ \
\ \ \ \ \ \ \ \ \ \ \ \ \ \ \ \ \ \ \ \ \ \ \ \ \ \ \ \ \ \ \ \ \ \ \ \ \ \
\ \ \ \ \ \ \ \ \ \ \ \ \ \ \ \ \ \ }
\end{align}%
Because of $\mathbf{\bar{N}}$ is parallel to $\mathbf{N},$ then%
\begin{equation}
\left\langle \mathbf{\bar{N}},\mathbf{T}\right\rangle =0,\left\langle
\mathbf{\bar{T}},\mathbf{N}\right\rangle =0.
\end{equation}%
Since $\left( \gamma ,\bar{\gamma}\right) $ is a Berrand pair in the
equiform geometry of $G_{3}^{1}$, then from Theorem $5.1$, we have%
\begin{equation}
\mathcal{K=}0\text{ and }\mathcal{\bar{K}=}0.
\end{equation}%
After substituting $(5.4)$ and $(5.5)$ into $(5.3)$, we get
\begin{equation}
\left\langle \mathbf{\bar{T}},\mathbf{T}\right\rangle ^{\prime }=0.
\end{equation}%
In the light of $(5.6)$ the angle between $\mathbf{\bar{T}},\mathbf{T}$ is
constant. Thus this completes the proof.
\end{proof}

\begin{corollary}
Let $\gamma (s):I\rightarrow G_{3}^{1}$ be a Bertrand curve in the equiform
geometry of $G_{3}^{1}.$ Then

(i) $\gamma $ is a weak equiform AW$(3)$-type but not a weak equiform AW$(2)$%
-type.

(ii) $\gamma $ is equiform AW$(3)$-type but not equiform AW$(1)$ and AW$(2)$%
-types.
\end{corollary}

\section{Examples}

We consider some examples (timelike and spacelike curves \cite{11,12}) which
characterize equiform general (circular) helices with respect to the Frenet
frame $\left\{ \mathbf{T},\mathbf{N},\mathbf{B}\right\} $ in the equiform
geometry of $G_{3}^{1}$ which satisfy some conditions of equiform curvatures
($\mathcal{K=K}(s),\mathcal{T=T}(s);~\mathcal{K=}const.\neq 0,\mathcal{T=}%
const.\neq 0;~\mathcal{K=}const.\neq 0,\mathcal{T=}0$).

\begin{example}
Consider the equiform \textbf{timelike} general helix $\mathbf{r}%
:I\longrightarrow G_{3}^{1},I\subseteq \mathbb{R}$ parameterized by the arc
length $s$ with differential form $ds=dx,$ given by
\begin{equation*}
\mathbf{r}(x)=(x,y(x),z(x)),
\end{equation*}%
where
\begin{eqnarray*}
x(s) &=&s, \\
y(s) &=&\frac{e^{-as}}{\left( a^{2}-b^{2}\right) ^{2}}\left( \left(
a^{2}+b^{2}\right) \cosh \left( bs\right) +2ab\sinh \left( bs\right) \right)
, \\
z(s) &=&\frac{e^{-as}}{\left( a^{2}-b^{2}\right) ^{2}}\left( 2ab\cosh \left(
bs\right) +\left( a^{2}+b^{2}\right) \sinh \left( bs\right) \right) ; \\
a,b&\in&\mathbb{R}-\left\{ 0\right\}.
\end{eqnarray*}%
The corresponding derivatives of $\mathbf{r}$\textbf{\ }are as follows
\begin{eqnarray*}
\mathbf{r}^{\prime } &=&\left( 1,\frac{-e^{-as}}{\left( a^{2}-b^{2}\right) }%
\left( a\cosh \left( bs\right) +b\sinh \left( bs\right) \right) ,\frac{%
e^{-as}}{\left( b^{2}-a^{2}\right) }\left( b\cosh \left( bs\right) +a\sinh
\left( bs\right) \right) \right) , \\
\mathbf{r}^{\prime \prime } &=&\left( 0,e^{-as}\cosh \left( bs\right)
,e^{-as}\sinh \left( bs\right) \right) , \\
\mathbf{r}^{\prime \prime \prime } &=&\left( 0,e^{-as}\left( -a\cosh \left(
bs\right) +b\sinh \left( bs\right) \right) ,e^{-as}\left( b\cosh \left(
bs\right) -a\sinh \left( bs\right) \right) \right) .
\end{eqnarray*}

First of all, we find that the tangent vector of $\mathbf{r}$ has the form
\begin{eqnarray*}
\mathbf{e}_{1} &=&\left( x^{\prime },y^{\prime },z^{\prime }\right) \\
&=&\left( 1,\frac{-e^{-as}}{\left( a^{2}-b^{2}\right) }\left( a\cosh \left(
bs\right) +b\sinh \left( bs\right) \right) ,\frac{e^{-as}}{\left(
b^{2}-a^{2}\right) }\left( b\cosh \left( bs\right) +a\sinh \left( bs\right)
\right) \right) .
\end{eqnarray*}

Then the two normals (normal and binormal) of the curve are, respectively%
\begin{eqnarray*}
\mathbf{e}_{2} &=&\left( 0,\cosh \left( bs\right) ,\sinh \left( bs\right)
\right) , \\
\mathbf{e}_{3} &=&\left( 0,\sinh \left( bs\right) ,\cosh \left( bs\right)
\right) ;~~\det [\mathbf{e}_{1},\mathbf{e}_{2},\mathbf{e}_{3}]=1.
\end{eqnarray*}

Thus the computations of the coordinate functions of $\mathbf{r}$ lead to%
\begin{equation*}
\kappa =e^{-as},~\tau =b~.
\end{equation*}

From the equiform Frenet formulas $(3.5)$ we can express vector fields $%
\mathbf{T},\mathbf{N},\mathbf{B}$ as follows
\begin{eqnarray*}
\mathbf{T} &=&\left( e^{as},\frac{-1}{\left( a^{2}-b^{2}\right) }\left(
a\cosh \left( bs\right) +b\sinh \left( bs\right) \right) ,\frac{1}{\left(
b^{2}-a^{2}\right) }\left( b\cosh \left( bs\right) +a\sinh \left( bs\right)
\right) \right) , \\
\mathbf{N} &=&\left( 0,e^{as}\cosh \left( bs\right) ,e^{as}\sinh \left(
bs\right) \right) , \\
\mathbf{B} &=&\left( 0,e^{as}\sinh \left( bs\right) ,e^{as}\cosh \left(
bs\right) \right) ,
\end{eqnarray*}%
respectively.
In the light of this, the equiform curvatures are given by%
\begin{equation*}
~\mathcal{K}=ae^{as},\mathcal{T}=-be^{as}.
\end{equation*}
\end{example}
\begin{center}
\begin{figure}[h]
\centering
\includegraphics[width=5cm]{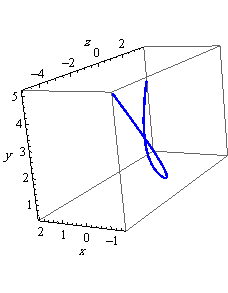} \label{fig:timelikeg2}
\caption{Equiform timelike general helix with $\mathcal{K}(s)=e^{s},\mathcal{%
T}(s)=2e^{s}$.}
\end{figure}
\end{center}

\begin{example}
Let $\mathbf{r}:I\longrightarrow G_{3}^{1},I\subseteq \mathbb{R}$ be the equiform
\textbf{spacelike} general helix, given by
\begin{equation*}
\mathbf{r}(x)=(x,y(x),z(x)),
\end{equation*}%
where
\begin{eqnarray*}
x(s) &=&s, \\
y(s) &=&\frac{e^{-as}}{\left( a^{2}-b^{2}\right) ^{2}}\left( 2ab\cosh \left(
bs\right) +\left( a^{2}+b^{2}\right) \sinh \left( bs\right) \right) , \\
z(s) &=&\frac{e^{-as}}{\left( a^{2}-b^{2}\right) ^{2}}\left( \left(
a^{2}+b^{2}\right) \cosh \left( bs\right) +2ab\sinh \left( bs\right) \right)
; \\
a,b &\in &\mathbb{R}-\left\{ 0\right\} .
\end{eqnarray*}%
For the coordinate functions of $\mathbf{r}$, we have
\begin{eqnarray*}
\mathbf{r}^{\prime } &=&\left( 1,\frac{e^{-as}}{\left( b^{2}-a^{2}\right) }%
\left( b\cosh \left( bs\right) +a\sinh \left( bs\right) \right) ,\frac{%
-e^{-as}}{\left( a^{2}-b^{2}\right) }\left( a\cosh \left( bs\right) +b\sinh
\left( bs\right) \right) \right) , \\
\mathbf{r}^{\prime \prime } &=&\left( 0,e^{-as}\sinh \left( bs\right)
,e^{-as}\cosh \left( bs\right) \right) , \\
\mathbf{r}^{\prime \prime \prime } &=&\left( 0,e^{-as}\left( b\cosh \left(
bs\right) -a\sinh \left( bs\right) \right) ,e^{-as}\left( b\sinh \left(
bs\right) -a\cosh \left( bs\right) \right) \right) .
\end{eqnarray*}%
Also, the associated trihedron is given by
\begin{eqnarray*}
\mathbf{e}_{1} &=&\left( 1,\frac{e^{-as}}{\left( b^{2}-a^{2}\right) }\left(
b\cosh \left( bs\right) +a\sinh \left( bs\right) \right) ,\frac{-e^{-as}}{%
\left( a^{2}-b^{2}\right) }\left( a\cosh \left( bs\right) +b\sinh \left(
bs\right) \right) \right) , \\
\mathbf{e}_{2} &=&\left( 0,\sinh \left( bs\right) ,\cosh \left( bs\right)
\right) , \\
\mathbf{e}_{3} &=&\left( 0,-\cosh \left( bs\right) ,-\sinh \left( bs\right)
\right) .~
\end{eqnarray*}

The curvature and torsion of this curve are
\begin{equation*}
\kappa =e^{-as},~\tau =-b~.~
\end{equation*}

Furthermore, the tangent, normal and binormal vector fields in the equiform
geometry of $G_{3}^{1}$ are obtained as follows
\begin{eqnarray*}
\mathbf{T} &=&\left( e^{as},\frac{1}{\left( b^{2}-a^{2}\right) }\left(
b\cosh \left( bs\right) +a\sinh \left( bs\right) \right) ,\frac{-1}{\left(
a^{2}-b^{2}\right) }\left( a\cosh \left( bs\right) +b\sinh \left( bs\right)
\right) \right) , \\
\mathbf{N} &=&\left( 0,e^{as}\sinh \left( bs\right) ,e^{as}\cosh \left(
bs\right) \right) , \\
\mathbf{B} &=&\left( 0,-e^{as}\cosh \left( bs\right) ,-e^{as}\sinh \left(
bs\right) \right) ,
\end{eqnarray*}%
respectively.

The equiform curvatures of $\mathbf{r}$ are
\begin{equation*}
~\mathcal{K}=ae^{as},\mathcal{T}=-be^{as}.
\end{equation*}
\end{example}

\begin{center}
\begin{figure}[h]
\centering
\includegraphics[width=6.5cm]{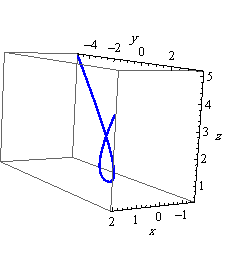} \label{fig:spacelikeg2}
\caption{Equiform spacelike general helix with $\mathcal{K}(s)=e^{s},%
\mathcal{T}(s)=-2e^{s}$.}
\end{figure}
\end{center}

\begin{example}
In this example, let us consider the equiform timelike\textbf{\ circular}
helix $\mathbf{r}:I\longrightarrow G_{3}^{1}$ given by
\begin{equation*}
\mathbf{r}(x)=(x,y(x),z(x)),
\end{equation*}%
where
\begin{eqnarray*}
x(s) &=&s, \\
y(s) &=&\frac{a^{3}s}{b\left( b^{2}-a^{2}\right) }\left( b\sinh \left( \frac{%
b}{a}\ln (as)\right) -a\cosh \left( \frac{b}{a}\ln (as)\right) \right) , \\
z(s) &=&\frac{a^{3}s}{b\left( b^{2}-a^{2}\right) }\left( b\cosh \left( \frac{%
b}{a}\ln (as)\right) -a\sinh \left( \frac{b}{a}\ln (as)\right) \right) ; \\
a,b &\in &\mathbb{R}-\left\{ 0\right\} .
\end{eqnarray*}%
For this curve, the equiform vector fields are obtained as follows
\begin{eqnarray*}
\mathbf{T} &=&\left( \frac{s}{a},\frac{as}{b}\cosh \left( \frac{b}{a}\ln
(as)\right) ,\frac{as}{b}\sinh \left( \frac{b}{a}\ln (as)\right) \right) , \\
\mathbf{N} &=&\left( 0,\frac{s}{a}\sinh \left( \frac{b}{a}\ln (as)\right) ,%
\frac{s}{a}\cosh \left( \frac{b}{a}\ln (as)\right) \right) , \\
\mathbf{B} &=&\left( 0,\frac{s}{a}\cosh \left( \frac{b}{a}\ln (as)\right) ,%
\frac{s}{a}\sinh \left( \frac{b}{a}\ln (as)\right) \right) ,
\end{eqnarray*}%
respectively.

It follows that%
\begin{equation*}
\mathcal{K}=\frac{1}{a},\mathcal{T}=\frac{-b}{a^{2}}.
\end{equation*}
\end{example}

\begin{center}
\begin{figure}[h]
\centering
\includegraphics[width=6cm]{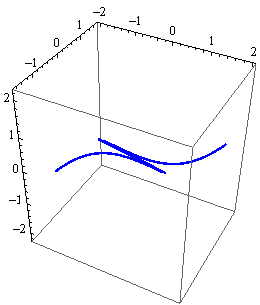} \label{fig:circulart}
\caption{Equiform timelike circular helix with $\mathcal{K}=\frac{1}{a},%
\mathcal{T}=\frac{-b}{a^{2}}$.}
\end{figure}
\end{center}

\begin{example}
Let the equiform \textbf{spacelike }circular helix $\mathbf{r}%
:I\longrightarrow G_{3}^{1},I\subseteq \mathbb{R}$ in the form
\begin{equation*}
\mathbf{r}(x)=(x,y(x),z(x)),
\end{equation*}%
where
\begin{eqnarray*}
x(s) &=&s, \\
y(s) &=&\frac{a^{3}s}{b\left( b^{2}-a^{2}\right) }\left( b\cosh \left( \frac{%
b}{a}\ln (as)\right) -a\sinh \left( \frac{b}{a}\ln (as)\right) \right) , \\
z(s) &=&\frac{a^{3}s}{b\left( b^{2}-a^{2}\right) }\left( b\sinh \left( \frac{%
b}{a}\ln (as)\right) -a\cosh \left( \frac{b}{a}\ln (as)\right) \right) ; \\
a,b &\in &\mathbb{R}-\left\{ 0\right\} .~
\end{eqnarray*}%
Here, the equiform differntial vectors are respectively, as follows
\begin{eqnarray*}
\mathbf{T} &=&\left( \frac{s}{a},\frac{as}{b}\sinh \left( \frac{b}{a}\ln
(as)\right) ,\frac{as}{b}\cosh \left( \frac{b}{a}\ln (as)\right) \right) , \\
\mathbf{N} &=&\left( 0,\frac{s}{a}\cosh \left( \frac{b}{a}\ln (as)\right) ,%
\frac{s}{a}\sinh \left( \frac{b}{a}\ln (as)\right) \right) , \\
\mathbf{B} &=&\left( 0,-\frac{s}{a}\sinh \left( \frac{b}{a}\ln (as)\right) ,-%
\frac{s}{a}\cosh \left( \frac{b}{a}\ln (as)\right) \right) .
\end{eqnarray*}%
Equiform curvature and equiform torsion are calculated as follows%
\begin{equation*}
\mathcal{K}=\frac{1}{a},\mathcal{T}=\frac{b}{a^{2}}.
\end{equation*}
\end{example}

\begin{center}
\begin{figure}[h]
\centering
\includegraphics[width=5cm]{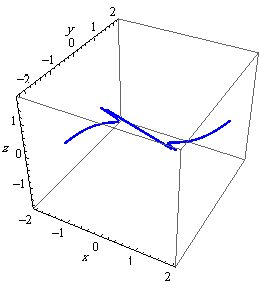} \label{fig:circulars}
\caption{Equiform spacelike circular helix with $\mathcal{K}=\frac{1}{a},%
\mathcal{T}=\frac{b}{a^{2}}$.}
\end{figure}
\end{center}
\begin{example}
If we consider the equiform \textbf{timelike} isotropic logarithmic spiral $%
\mathbf{r}:I\longrightarrow G_{3}^{1},I\subseteq \mathbb{R}$ parameterized by the arc
length $s$ with differential form $ds=dx,$ given by%
\begin{equation*}
\mathbf{r}(x)=(x,y(x),0),
\end{equation*}%
where
\begin{eqnarray*}
x(s) &=&s, \\
y(s) &=&\frac{as+b}{a^{2}}\left( \ln (as+b)-1\right) , \\
z(s) &=&0; \\
a,b &\in &\mathbb{R}-\left\{ 0\right\} .
\end{eqnarray*}%
For this curve, we get%
\begin{eqnarray*}
\mathbf{r}^{\prime } &=&\left( 1,\frac{\ln (as+b)}{a},0\right) , \\
\mathbf{r}^{\prime \prime } &=&\left( 0,\frac{1}{as+b},0\right) , \\
\mathbf{r}^{\prime \prime \prime } &=&\left( 0,\frac{-a}{\left( as+b\right)
^{2}},0\right) ,
\end{eqnarray*}%
and
\begin{eqnarray*}
\mathbf{e}_{1} &=&\left( 1,\frac{\ln (as+b)}{a},0\right) , \\
\mathbf{e}_{2} &=&\left( 0,1,0\right) , \\
\mathbf{e}_{3} &=&\left( 0,0,1\right) ;~\kappa =\frac{1}{as+b},~\tau =0.
\end{eqnarray*}%
In this case, equiform Frenet vectors and equiform curvatures are as follows
\begin{eqnarray*}
\mathbf{T} &=&\left( as+b,\frac{\left( as+b\right) \ln (as+b)}{a},0\right) ,
\\
\mathbf{N} &=&\left( 0,as+b,0\right) , \\
\mathbf{B} &=&\left( 0,0,as+b\right) ,~\mathcal{K}=a,\mathcal{T}=0.
\end{eqnarray*}%
respectively.
\end{example}
\begin{center}
\begin{figure}[h]
\centering
\includegraphics[width=5cm]{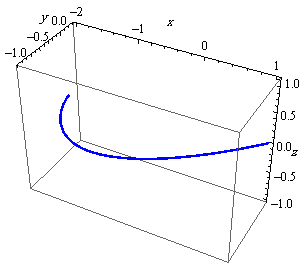} \label{fig:spiral11}
\caption{Equiform timelike isotropic logarithmic spiral with $\mathcal{K}%
(s)=1,\mathcal{T}(s)=0$.}
\end{figure}
\end{center}

From aforementioned calculations, according to (\textbf{Proposition }$%
\mathbf{4.2}$\textbf{\ and Theorems }$\mathbf{4.1-4.3}$), examples $1-4$ are
not characterize curves of equiform AW$(k)$, weak equiform AW$(2)$ and weak
equiform AW$(3)$-types. On the other hand, the last example shows that the
curve is of equiform AW$(2)$ and AW$(3)$-types and it is not of equiform AW$%
(1)$-type. Also, it is of weak equiform AW$(2)$ and not of weak equiform AW$%
(3)$-types.

\section{Conclusion}

In this paper, we have considered some special curves of equiform AW$(k)$-type of the pseudo-Galilean $3$-space. Also, using the equiform curvature
conditions of these curves, the necessary and sufficient conditions for them
to be equiform AW$(k)$ and weak equiform AW$(k)$-types are given.
Furthermore, several examples to confirm our main results have been given
and illustrated.


\begin{thebibliography}{99}
\bibitem{1} I. Yaglom, A simple non-Euclidean geometry and its physical
basis, Springer-Verlag, in New York, 1979.

\bibitem{2} B. J. Pavkovi\'{c}, Equiform Geometry of Curves in the Isotropic
Spaces $I_{3}^{1}$ and $I_{3}^{2}$, Rad JAZU, (1986), 39-44.

\bibitem{3} B. J. Pavkovi\'{c} and I. Kamenarovi\'{c}, The Equiform
Differential Geometry of Curves in the Galilean Space $G_{3}$, Glasnik Mat.
Vol 22(42)1987), 449-457.

\bibitem{4} K. Arslan, A. West, Product submanifolds with pointwise 3-planar normal sections, Glasg. Math. J. 37 (1)(1995), 73-81.

\bibitem{5} K. Arslan and C. \"{O}zg\"{u}r, Curves and surfaces of AW$(k)$%
-type, Geometry and topology of
submanifolds,IX(Valenciennes/Lyan/Leuven,1997), 21-26, World. Sci.
Publishing, River Edge, NJ, 1999.

\bibitem{6} M. K\"{u}lahci, M. Bektas and M. Erg\"{u}t, On harmonic
curvatures of null curves of \ the AW$(k)$-type in Lorentzian space, Z.
Naturforsch, 63 a (2008), 248-252.

\bibitem{7} M. K\"{u}lahci and M. Erg\"{u}t, Bertrand curves of AW$(k)$-type
in Lorentzian space, Non Linear Analysis, 70(2009), 1725-1731.

\bibitem{8} M. K\"{u}lahci, A.O. \"{O}\u{g}renmi\c{s} and M. Erg\"{u}t, New
characterizations of curves in the Galilean space $G_{3}$, International
Journal of physical and Mathematical Sciences, 1(2010), 49-57.

\bibitem{9} C. \"{O}zg\"{u}r and F. Gezgin, On some curves of AW(k)-type,
Differ. Geom. Dyn. Syst., 7(2005), 74-80.

\bibitem{10} D. W. Yoon, General Helices of AW$(k)$-Type in the Lie Group,
Journal of Applied Mathematics, Article ID 535123, (2012), 1-10.

\bibitem{11} Z. Erjavec and B. Divjak, The equiform differential geometry of
curves in the pseudo-Galilean space, Math. Communications, 13(2008), 321-332.

\bibitem{12} Z. Erjavec, On Generalization of Helices in the Galilean and
the Pseudo-Galilean Space, Journal of Mathematics Research, 6(3)(2014),
39-50.

\bibitem{13} B. Divjak, The General Solution of the Frenet's System of
Differential Equations for Curves in the Pseudo-Galilean Space $G_{3}^{1}$,
Math. Communications, 2(1997), 143-147.

\bibitem{14} B. Divjak, Geometrija pseudogalilejevih prostora, Ph. D.
thesis, University of Zagreb, 1997.

\bibitem{15} B. Divjak, Curves in pseudo-Galilean geometry, Annales Univ.
Sci. Budapest 41(1998), 117-128.
\end{thebibliography}
\end{document}